\documentclass[letterpaper, 10 pt,  conference]{ieeeconf}  % Comment this line out
\IEEEoverridecommandlockouts                              % This command is only
\usepackage{graphics} % for pdf, bitmapped graphics files
\usepackage{epsfig} % for postscript graphics files
\usepackage{mathptmx} % assumes new font selection scheme installed
\usepackage{times} % assumes new font selection scheme installed
\usepackage{amsmath} % assumes amsmath package installed

\usepackage{amssymb}  % assumes amsmath package installed
\usepackage{subfigure}

\newcommand{\R}{{\mathbb R}}
\newtheorem{theorem}{Theorem}

\title{\LARGE \bf
Frequency Domain Properties and Fundamental Limits of Buffer-Feedback Regulation in Biochemical Systems}

\author{Edward J. Hancock and Jordan Ang% <-this % stops a space
\thanks{Edward J. Hancock is with the School of Mathematics and Statistics \& The Charles Perkins Centre,
        The University of Sydney, NSW, 2006, Australia
        {\tt\small edward.hancock@sydney.edu.au}}%
\thanks{Jordan Ang is with the Department of Chemical and Physical Sciences, The University of Toronto, Mississauga,
ON L5L1C6, Canada and the Department of Immunology, The University of Toronto,
Toronto, ON, M5S1A8, Canada {\tt\small j.ang@utoronto.ca}}%
\thanks{Part of the work in this paper was presented at the 2017 IEEE Conference on Decision and Control on Dec 12-15, 2017, in Melbourne, Australia.}
}

\begin{document}

\maketitle
\thispagestyle{empty}
\pagestyle{empty}%%%%%%%%%%%%%%%%%%%%%%%%%%%%%%%%%%%%%%%%%%%%%%%%%%%%%%%%%%%%%%%%%%%%%%%%%%%%%%%%
\begin{abstract}

Feedback regulation in biochemical systems is fundamental to homeostasis, with failure causing disease or death. Recent work has found that cooperation between feedback and buffering---the use of reservoirs of molecules to maintain molecular concentrations---is often critical for biochemical regulation, and that buffering can act as a derivative or lead controller. However, buffering differs from derivative feedback in important ways: it is not typically limited by stability constraints on the parallel feedback loop, for some signals it acts instead as a low-pass filter, and it can change the location of disturbances in the closed-loop system. Here, we propose a frequency-domain framework for studying the regulatory properties of buffer-feedback systems. We determine standard single-output closed-loop transfer functions and discuss loop-shaping properties. We also derive novel fundamental limits for buffer-feedback regulation, which show that buffering and removal processes can reduce the fundamental limits on feedback regulation. We apply the framework to study the regulation for glycolysis (anaerobic metabolism) with creatine phosphate buffering. 
\end{abstract}

%%%%%%%%%%%%%%%%%%%%%%%%%%%%%%%%%%%%%%%%%%%%%%%%%%%%%%%%%%%%%%%%%%%%%%%%%%%%%%%%
\section{INTRODUCTION}

Discovering the universal principles of biological robustness is important for understanding the principles of life, treating complex diseases, and creating \textit{de novo} synthetic biology \cite{KIT04}. A ubiquitous mechanism for robustness in biology is feedback regulation, which has received extensive study \cite{ALO07,CHABD11,IGLI09,KIT04,SAU17,YIHSD00}. However, universal theories of robustness need to be extended beyond feedback regulation \cite{KIT04}. Buffering---the use of reservoirs of molecules to maintain molecular concentrations---is another important and ubiquitous mechanism for robust regulation in biology \cite{HANAPS17,SHE13}. In contrast to one another, feedback acts via a biological actuator (e.g., a regulated enzyme) \cite{HANAPS17,SHE13}, while buffering, as we consider it in this paper, is a non-actuated form of regulation occurring via Le Chatelier-driven reactions. Examples include pH buffering, ATP energy buffering via creatine phosphate, reservoirs established through cell compartmentalization and other spatial barriers \cite{BERS10,HANAPS17,SHE13}. Other forms of buffering include small synthetic biology networks designed to increase component modularity \cite{VECNS08}. Outside of biology, combined regulation through feedback and buffering tanks is common in industrial chemical processes \cite{FAAS03}.

In recent work, we have found that cooperation between feedback and buffering mechanisms are often critical for biochemical regulation, and that buffering can act as a derivative or lead controller \cite{HANAPS17}. However, buffering differs from derivative feedback: it is not typically limited by any stability constraints on the parallel feedback loop \cite{HANAPS17}, such as from delays or autocatalysis (where the system's output is necessary to catalyze its own production - see \cite{CHABD11,SIAEA17}). Further, for some signals it acts instead as a low-pass filter, closer to a two degrees-of-freedom controller, and it can change the location of disturbances in the closed-loop system. As a result, the analysis requires a modified approach from standard control theory in order to take the structure of buffer-feedback regulation into account. Due to the regulatory constraints arising from autocatalytic networks in biology \cite{BUZD10, CHABD11,MOTEA10, SIAEA17}, it is important for this modified approach to incorporate fundamental limits on feedback regulation. 

In this paper, we study frequency-domain properties of buffer-feedback regulation. We create a general framework to extend our recent work from minimal models \cite{HANAPS17} to higher-order single-output models. We determine standard closed-loop transfer functions for buffer-feedback regulatory structures and describe their loop-shaping properties.  We also determine novel fundamental limits on buffer-feedback regulation that are functions of separate open-loop contributions from feedback and buffering. These limits show that buffering and removal processes can reduce fundamental constraints on feedback regulation. Finally, we apply the developed framework to the analysis of glycolysis (anaerobic metabolism) with creatine phosphate buffering: a buffered, autocatalytic biochemical system. 

The structure of the paper is as follows: Section \ref{sect:pf} introduces the problem formulation, Section \ref{sect:tf} determines the closed-loop transfer functions, Section \ref{sect:loop} discusses loop-shaping properties of buffer-feedback systems, Section \ref{sect:constraint} determines fundamental constraints on the closed loop systems, and Section \ref{sect:exam} applies the framework to the example of glycolysis.

\section{Problem Formulation}\label{sect:pf}

In this section, we propose a general model for buffer-feedback regulatory systems. We first propose an input-output model, and then include buffering and feedback. An example of the final model is diagrammed in Figure~\ref{fig:top}. 

\subsection{Input-Output Model}

We consider the single-output process model
\begin{equation}\label{eq:min_mod}
\begin{aligned}
	\dot{y}&=p_y(y,z,u_h)-r_y(y,z,u_h)+u_b+d_y(t)\\
	\dot{z}&=p_z(y,z,u_h)-r_z(y,z,u_h)+d_z(t)\\
\end{aligned}
\end{equation}
where $y:\R\to\R$ is the regulated species concentration, $z:\R\to\R^{n-1}$ contains $n-1$ intermediate species concentrations, $u_h:\R\to\R$ is the feedback control input, $u_b:\R\to\R$ is the buffering control input,  $p_y:\R^n\times\R\to\R$ and $p_z:\R^n\times\R\to\R^{n-1}$ are the production rates of $y$ and $z$ respectively, $r_y:\R^n\times\R\to\R$ and $r_z:\R^n\times\R\to\R^{n-1}$ are the removal rates at $y$ and $z$ respectively, and $d_y:\R\to\R$ and $d_z:\R\to\R$ are the disturbances at $y$ and $z$ respectively.

We are interested in studying the behaviour of the regulatory system about a set point $\bar{y}$. For example, we could have $\bar{y}=3mM$ ATP for energy regulation. For the nominal case ($d_y=d_z=0$), we therefore assume that the steady-state inputs $u_h=\bar{u}_h(\bar{y})$ and $u_b=\bar{u}_b(\bar{y})$ are a function of $\bar{y}$, and that for \eqref{eq:min_mod} there is a nominal closed-loop steady state $(\bar{y},\bar{z})$ of interest. 

\begin{figure}
\centering
    \includegraphics[width=0.75\columnwidth]{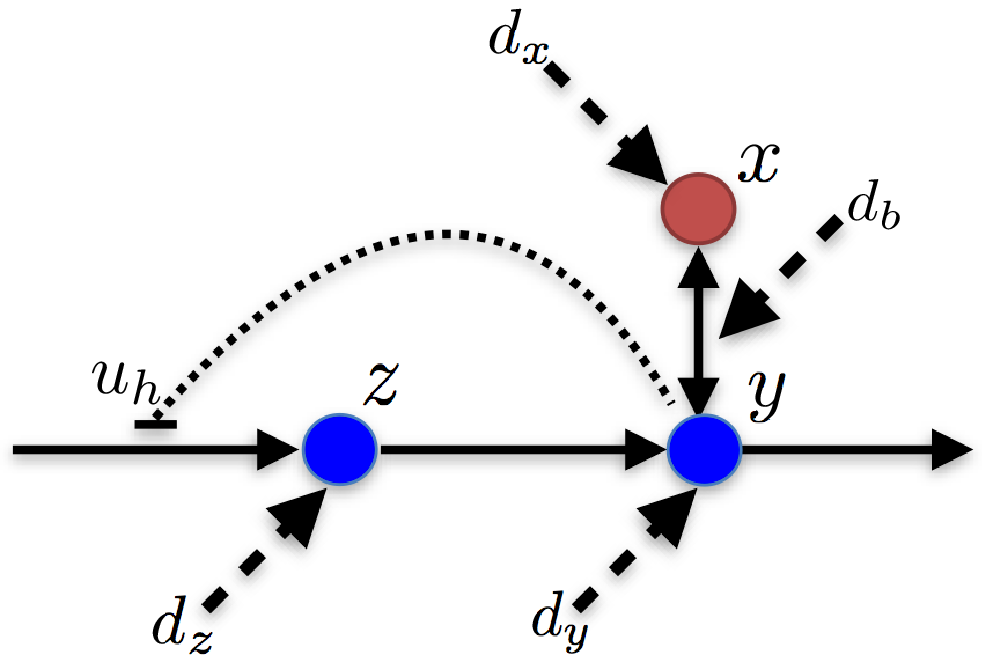}
\caption{An Example of a Buffer-Feedback Regulatory System, where $y$ represents the regulated species and $x$ represents the buffering species.}
\label{fig:top}
\end{figure}

We define the deviations about the nominal steady state as
\begin{equation*}
\begin{aligned}
	\Delta y =y-\bar{y},\quad \Delta z =z-\bar z,\quad \Delta u_h =u_h-\bar{u}_h,\quad \Delta u_b =u_b-\bar{u}_b\\
\end{aligned}
\end{equation*}
Linearizing \eqref{eq:min_mod}, we obtain
\begin{equation}\label{eq:min_mod_ol}
\begin{aligned}
\Delta\dot{y}&=A_{yy}\Delta y+A_{yz}\Delta z+B_{yh}\Delta u_h+\Delta u_b+d_y\\
\Delta\dot{z}&=A_{zy}\Delta y+A_{zz}\Delta z+B_{zh}\Delta u_h+d_z\\
\end{aligned}
\end{equation}
where 
\begin{equation*}
\begin{aligned}
A_{yy}&=\frac{\partial}{\partial y}(p_y-r_y),&A_{yz}&=\frac{\partial}{\partial z}(p_y-r_y),\\
A_{zy}&=\frac{\partial}{\partial y}(p_z-r_z),&A_{zz}&=\frac{\partial}{\partial z}(p_z-r_z),\\
B_{yh}&=\frac{\partial}{\partial u_h}(p_y-r_y),&B_{zh}&=\frac{\partial}{\partial u_h}(p_z-r_z)
\end{aligned}
\end{equation*}
all of which are evaluated at the nominal steady state $(\bar{y},\bar{z})$ and nominal input $(\bar{u}_h,\bar{u}_b)$ for the open-loop system. 

\subsection{Process Model with Buffering}

We next include a buffering system, expanding \eqref{eq:min_mod} to be
\begin{equation}\label{eq:min_mod_buff}
\begin{aligned}
	\dot{y}=&p_y(z,u_h)-r_y(y,u_h)+d_y(t)\mathbf{-g_y(y,x)+g_x(y,x)-d_b(t)}\\
	\dot{z}=&p_z(y,z,u_h)-r_z(y,z,u_h)+d_z(t)\\
	\mathbf{\dot{x}}=&\mathbf{g_y(y,x)-g_x(y,x)-r_x(x)+d_x(t)+d_b(t)}
\end{aligned}
\end{equation}
where $x$ is the buffering species concentration, $r_x:\R\to\R$ is the removal rate of $x$, $d_x:\R\to\R$ represents a disturbance at the buffering species $x$, the (lumped) buffering reactions are $g_y:\R^n\to\R$ and $g_x:\R^n\to\R$ with $g_y$ representing $y\!\rightarrow\!x$ and $g_x$ representing $x\!\rightarrow\!y$, and $d_b:\R\to\R$ represents a disturbance in the buffering reactions $g_y$ and $g_x$. The additional terms to those in \eqref{eq:min_mod} are highlighted in bold.

We define the case $r_x=0$ as lossless buffering and the case $r_x\ne 0$ as dissipative buffering or buffering with a removal process. Consistent with the no buffering case, we assume that for \eqref{eq:min_mod} there is a nominal ($d_y=d_z=d_b=d_x=0$) steady state $(\bar{y},\bar{z},\bar{x})$ of interest at the nominal feedback input $u_h=\bar{u}_h(\bar{y})$ such that
\begin{equation*}
\begin{aligned}
&\bar{u}_b=-g_y(\bar y,\bar x)+g_x(\bar y,\bar x)\\
&g_y(\bar y,\bar x)=g_x(\bar y,\bar x)+r_x(\bar x)
\end{aligned}
\end{equation*}
where $\bar{u}_b=0$ for the lossless case.

We define the deviation for the buffering species
\begin{equation*}
\begin{aligned}
	\Delta x =x-\bar x\\
\end{aligned}
\end{equation*}
and linearize \eqref{eq:min_mod_buff} to obtain
\begin{equation}\label{eq:min_mod_lin}
\begin{aligned}
\Delta\dot{y}&=A_{yy}\Delta y+A_{yz}\Delta z-\sigma_y \Delta y +\sigma_x\Delta x+B_{yh}\Delta u_h+d_y-d_b\\
\Delta\dot{z}&=A_{zy}\Delta y+A_{zz}\Delta z+B_{zh}\Delta u_h+d_z\\
\Delta\dot{x}&=\sigma_y \Delta y -(\sigma_x+a_{xx} )\Delta x+d_x+d_b
\end{aligned}
\end{equation}
where 
\begin{equation*}
\begin{aligned}
\sigma_y&=\frac{\partial }{\partial y}(g_y-g_x),&\sigma_x&=\frac{\partial }{\partial x}(g_x-g_y),&a_{xx}&=\frac{\partial r_x}{\partial x}
\end{aligned}
\end{equation*}
which are evaluated at the nominal steady state $(\bar{y},\bar{z},\bar{x})$. 

\section{Transfer Functions}\label{sect:tf}

In this section, we determine a set of transfer functions between the system variables. 

\subsection{Process Model}

\begin{figure}
\centering
    \includegraphics[width=\columnwidth]{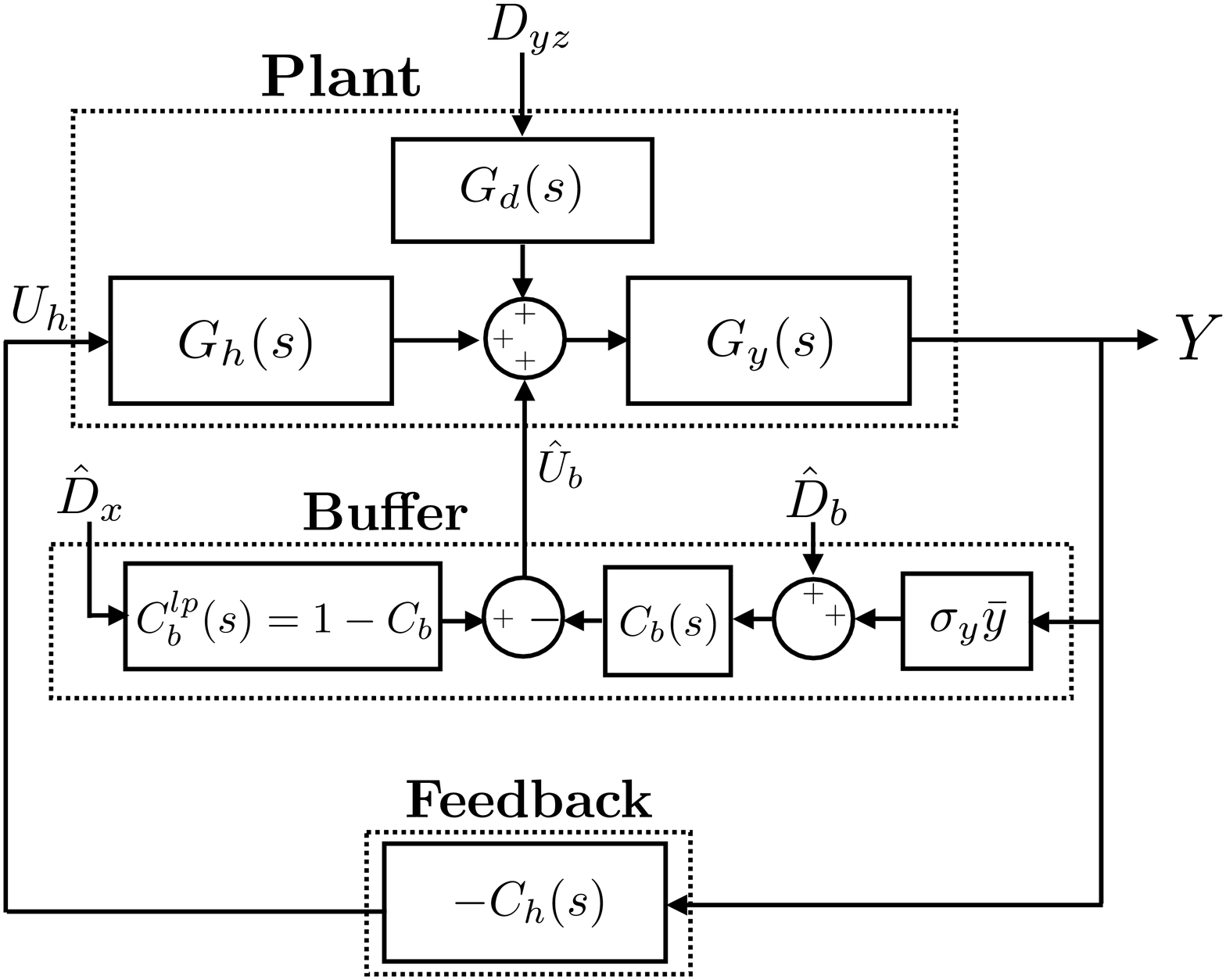}
\caption{Block Diagram of Buffer-Feedback Regulation described by Equation \eqref{eq:sf}. To simplify the layout, we use input $\hat{U}_b=\bar{y}U_b$ and disturbances $\hat{D}_b=\hat{d}_bD_b$, $\hat{D}_x=\hat{d}_xD_x$.}
\label{fig:BDPB}
\end{figure}

The frequency-domain representation of the process model in
\eqref{eq:min_mod_ol} is (see Appendix)
\begin{equation}\label{eq:ol_tf}
\begin{aligned}
Y(s)=&\frac{\hat{p}}{\bar{y}}\hat G_y(s)\hat G_h(s)U_h(s)+\frac{\hat{g}}{\bar{y}}\hat G_y(s)U_b(s)\\
&+\hat G_y(s)\hat G_d(s)\frac{\hat d_{yz}}{\bar{y}}D_{yz}(s)\\
\end{aligned}
\end{equation}
where 
\begin{equation*}
\begin{aligned}
&Y(s)=\frac{1}{\bar{y}}{\mathcal{L}\{\Delta y(t)\}},\quad \hat d_{yz}=\text{Diag}({\hat{d}_y,\hat{d}_z})\\
& D_{yz}(s)=\left(\frac{1}{\hat d_y}\mathcal{L}\{d_y(t)\}, \frac{1}{\hat d_z}\mathcal{L}\{d_z(t)\}\right)^T\\
\end{aligned}
\end{equation*}
are the scaled output and disturbances in the frequency-domain, $\mathcal{L}(\cdot)$ is the Laplace transform,  $s$ is the complex frequency,  
\begin{equation*}
\begin{aligned}
&U_h(s)=\frac{1}{\hat{p}}\mathcal{L}\{\Delta u_h(t)\},\quad &U_b(s)&=\frac{1}{\hat{g}}\mathcal{L}\{\Delta u_b(t)\}
\end{aligned}
\end{equation*}
are the Laplace transforms of the scaled feedback and buffer inputs respectively, and
\begin{equation*}
\begin{aligned}
\hat G_y (s)&=(s-A_{yy}-\hat G_zA_{zy})^{-1},&\hat G_h(s)&=B_{yh}+\hat G_zB_{zh}\\
\hat G_z (s)&= A_{yz}(sI_{n-1}-A_{zz})^{-1},&\hat G_d(s)&=[1,\, \hat G_z]\\
\end{aligned}
\end{equation*}
We initially keep the scaling factors separate from the transfer functions for two reasons: First, the relation to existing sensitivity analysis used in biological modelling \cite{KLIEA09} is more explicit. Second, the disturbance size may not be independent of feedback and buffering, e.g., due to molecular noise \cite{HANAPS17}.

\subsection{Closed Loop (Buffering and Feedback)}

The transfer function model for the closed-loop system \eqref{eq:min_mod_lin} is (see Appendix)
\begin{equation}\label{eq:cl_tf_par}
\begin{aligned}
Y=&\frac{\hat d_{yz}}{\bar{y}}S\hat G_y\hat G_dD_{yz}+\frac{\hat d_x}{\bar{y}}S\hat G_y(1-\hat C_b)D_x-\frac{\hat d_b}{\bar{y}}S\hat G_y\hat C_bD_b\\
S=&\frac{1}{1+\hat G_y\hat C_b\sigma_y+\hat G_y\hat G_h\hat C_h}
\end{aligned} 
\end{equation}
where $S(s)$ is the classical sensitivity function \cite{SKOP05}, the unscaled feedback is $\hat{p}U_h=-\hat C_h(s)\bar{y}Y$, and
\begin{equation*}
\begin{aligned}
\hat C_b(s)&= \frac{s+a_{xx}}{s+\sigma_x+a_{xx}}\\
D_x(s)&=\frac{1}{\hat d_x}\mathcal{L}\{d_x(t)\}, \quad D_b(s)=\frac{1}{\hat d_b}\mathcal{L}\{d_b(t)\}\\
\end{aligned}
\end{equation*}

\subsection{Scaled Transfer Functions}

The scaled transfer function model for the closed-loop system \eqref{eq:min_mod_lin}, as represented in Figure \ref{fig:BDPB}, is
\begin{equation}\label{eq:sf}
\begin{aligned}
Y&=S G_y  G_d D_{yz}+S G_y(1-C_b)\hat d_xD_x-S G_yC_b\hat d_bD_b\\
S&=\frac{1}{1+ G_yC_b \sigma_y\bar{y}+ G_yG_hC_h}\\
\end{aligned} 
\end{equation}
where 
\begin{equation*}
\begin{aligned}
G_y&=\frac{1}{\bar{y}}\hat G_y,\quad G_h= \hat G_h\hat{p},\quad  G_d= \hat{G}_d\hat d_{yz},\\
 C_h&=\frac{\bar{y}}{\hat{p}}\hat C_h,\quad C_b =\frac{\bar{y}}{\hat{g}}\hat C_b=\hat C_b
\end{aligned} 
\end{equation*}
where we set $\hat{g}=\bar{y}$ as $x$ in \eqref{eq:min_mod_buff} can be scaled by $\bar{y}$ (see Appendix). The corresponding loop transfer function $L(s)$ \cite{SKOP05} is
\begin{equation}\label{eq:ltf}
\begin{aligned}
L(s)&=L_b(s)+L_h(s) \\
L_b(s)&= G_yC_b \sigma_y\bar{y},\quad L_h(s)= G_yG_hC_h\\
S(s)&=\frac{1}{1+L(s)}\\
\end{aligned}
\end{equation}

\section{Loop-shaping Properties}\label{sect:loop}

In this section, we discuss the loop-shaping concepts for feedback-buffering regulatory systems, which generalizes \cite{HANAPS17} to a higher-order setting. These include the lead/derivative controller properties of buffers, the stabilizing effect of buffering, and the low-pass filtering property of buffering. We highlight both similarities and differences with traditional control structures.

\subsection{Buffers Act as Derivative or Lead Controllers}

We first show that that buffers can act as derivative or lead controllers. For the nominal case, the buffering is equivalent to feedback of the form $U_b=- C_b\sigma_yY$, where
\begin{equation}\label{eq:buff}
C_b\sigma_y=\sigma_y\frac{s+a_{xx}}{s+\sigma_x+a_{xx}}
\end{equation}
It can be noted that \eqref{eq:buff} is the standard form for a lead controller \cite{HANAPS17}. For the case that $a_{xx}=0$ (lossless buffering), we have an equivalence to derivative filtering, where
\begin{equation*}
\begin{aligned}
C_b\sigma_y= \frac{\sigma_y}{\sigma_x}\frac{s}{1+\frac{s}{\sigma_x}}
\end{aligned}
\end{equation*}
For $a_{xx}=0$, the buffering has no effect on the steady-state regulation, which is handled by the parallel feedback loop. Dissipative buffering may regulate the steady state as $C_b(0)\ne 0$. However, the addition of dissipative buffering can require an increase in the steady-state production rate $p_y$; this increase can also increase the size of disturbances $\hat{d}_{yz}$ and so can partially or fully negate the improvement in steady-state regulation.

\subsection{Buffers Can Stabilize Feedback}

We can observe that buffers have the ability to stabilize feedback due to two properties:  First, buffering has a parallel action to feedback and is not typically subject to the non-minimum phase constraints from delays or autocatalysis. For the open-loop transfer function \eqref{eq:ltf}, the RHP zeros and delays will typically be contained in the feedback plant $G_h$. As $\sigma_yC_b$ is independent of $ G_h$, it is also independent of its RHP zeros or delays. Second, a derivative or lead controller is known to help stabilize feedback by providing the controller with phase lead \cite{NIS04}.

\subsection{Buffers Can Also Act as Low-Pass Filters}

We also show that buffers can also act as a barrier between the disturbance and the output. The barrier is in the form of a low-pass filter for disturbances at $d_x$. We can observe this result by noting that
\begin{equation*}
C_b^{lp}=1-C_b=\frac{\sigma_x}{s+\sigma_x+a_{xx}}
\end{equation*}
is a low-pass filter.  For a smaller $\sigma_x+a_{xx}$, the filter has a lower bandwidth. Barrier placement allows buffer-feedback systems to move the location of disturbances between $d_y$, $d_x$, and $d_z$. For example, a disturbance acts at $d_x$ if there is a biological membrane with a transport mechanism (i.e., a barrier) between the regulated species and the disturbance, but would act directly at $d_y$ in the absence of such a barrier.

The combined derivative control and low-pass filtering properties of buffering is similar to a two degrees-of-freedom controller \cite{ASTM08,SKOP05} with an additional constraint $C_b^{lp}+C_b=1$.

\section{Fundamental Limits}\label{sect:constraint} 

In this section, we derive novel fundamental limits for buffer-feedback systems in terms of separate open-loop contributions from the feedback and buffering. We find that buffering and removal processes can reduce fundamental constraints on feedback.

Theorem \ref{th:bode} is an application of previous limits \cite{DOYFT90} that takes into account the structure of buffer-feedback systems. Theorem \ref{th:zero_integral} and \ref{th:zero_peak} are novel modifications to previous results, where zeros of open-loop feedback $L_h$ are studied instead of the loop transfer function $L$.

In the following section we refer to `closed-loop stability', where the system in Figure \ref{fig:BDPB}, with additional disturbances at $U_h,Y$ and state $X$, is internally stable. Also, in Theorem \ref{th:zero_integral} and \ref{th:zero_peak} we assume that $C_h$ has no unstable pole-zero cancellations with $G_h$ to ensure well-posed conditions for open-loop zeros of $L_h$.

\subsection{Bode's Sensitivity Integral}

We first determine the effect of buffering on Bode's sensitivity integral \cite{SKOP05,SERBG97,ASTM08}. We show that Bode's sensitivity integral is reduced by increasing forward buffering kinetic constant $\sigma_y$.
\begin{theorem}\label{th:bode}
Assuming proper, rational $C_h$, strictly proper $G_h$, and closed-loop stability, then Bode's sensitivity integral is
\begin{equation}\label{eq:bode}
\int_0^\infty \ln|S(i\omega)|d\omega = \pi\sum_k Re(p_k)-\frac{\pi}{2}\sigma_y
\end{equation}
where $p_k$ represents the unstable poles of $L(s)= G_yC_b \sigma_y\bar{y}+GC_h$.
\end{theorem}
\begin{proof}
We first note that $G_y(s)$ has a relative degree of one, which results in 
\begin{equation*}
\begin{aligned}
\lim_{s\to\infty} s G_y\bar{y}C_b \sigma_y&=\lim_{s\to\infty} \frac{1}{\bar{y}}\frac{s}{s-A_{yy}-\hat G_zA_{zy}}\bar{y}\frac{s+a_{xx}}{s+\sigma_x+a_{xx}}\sigma_y\\
&=\sigma_y
\end{aligned}
\end{equation*}
We also have $\lim_{s\to\infty} s G_yG_hC_h\to 0$ as $G_h$ is strictly proper and $C_h$ is proper. Thus
\begin{equation*}
\lim_{s\to\infty}sL(s)= \lim_{s\to\infty}sG_y\bar{y}C_b\sigma_y+\lim_{s\to\infty}sG_y G_hC_h=\sigma_y
\end{equation*}
Finally, as $L(s)$ is strictly proper, we have \cite{SERBG97}
\begin{equation*}
\begin{aligned}
\int_0^\infty \ln|S(i\omega)|d\omega &= \pi\sum_k Re(p_k)-\frac{\pi}{2} \lim_{s\to\infty}sL(s)\\
&=\pi\sum_k Re(p_k)-\frac{\pi}{2}\sigma_y
\end{aligned}
\end{equation*}
\end{proof}
Theorem \ref{th:bode} uses the unstable open-loop poles of $L$. However, we can further assume that all unstable poles in $L$ are poles in $L_h$. In this case, the contributions from the open-loop feedback $L_h$ and buffering $L_b$ separate on the RHS of \eqref{eq:bode}.

Theorem \ref{th:bode} shows that Bode's integral is reduced by the forward buffering rate rather than a ratio between $\sigma_y$ and $\sigma_x$.

\subsection{Weighted Sensitivity Integral}

We next study the integral constraints for the case where the open-loop feedback has an unstable zero, but the open-loop buffering has only stable zeros.
\begin{theorem}\label{th:zero_integral}
Assuming proper, rational $C_h$, and closed-loop stability, then
\begin{equation}\label{eq:zero_integral}
\begin{aligned}
\int_0^\infty \ln|S(i\omega)|w(z,\omega)d\omega =& \pi\ln\prod_k \left|\frac{p_k+z}{p_k-z}\right|-\pi\ln|1+L_b(z)|
\end{aligned}
\end{equation}
where $L_b(s)=G_yC_b \sigma_y\bar{y}$, $z$ represents the single unstable zero of $L_h(s)=GC_h$, $p_k$ represents the unstable poles of $L=L_b+L_h$, $z$ is distinct from all $p_k$, and 
\begin{equation*}
w(z,\omega)= \frac{2z}{z^2+\omega^2}
\end{equation*}
\end{theorem}
If $L$ has no RHP poles then 
\begin{equation*}
\begin{aligned}
\int_0^\infty \ln|S(i\omega)|w(z,\omega)d\omega =& -\pi\ln|1+L_b(z)|
\end{aligned}
\end{equation*}
\begin{proof}
We follow a modification of the proof in \cite{DOYFT90}. The sensitivity function $S$ can be factorized as $S=S_{ap}S_{mp}$ where $S_{mp}$ is a minimum phase transfer function and $S_{ap}$ is an all-pass transfer function, both unique up to sign \cite{DOYFT90}. We have \cite{DOYFT90}
\begin{equation*}
\begin{aligned}
\int_0^\infty \ln|S(i\omega)|w(z,\omega)d\omega =\pi\ln|S_{mp}(z)| \\
\end{aligned}
\end{equation*}
where the sensitivity function $S$ can be factorized as $S=S_{ap}S_{mp}$, $S_{mp}$ is a minimum phase transfer function and
\begin{equation*}
\begin{aligned}
S_{ap}=\prod_k\frac{s-p_k}{s+p_k}
\end{aligned}
\end{equation*}
is an all-pass transfer function, both unique up to sign \cite{DOYFT90}. For the case of unstable poles then
\begin{equation*}
\begin{aligned}
\ln|S_{mp}(z)|&=\ln|S_{ap}^{-1}(z)S(z)| =\ln\left|\prod_k \frac{p_k+z}{p_k-z}\right|+\ln\left|\frac{1}{1+L_b(z)}\right|
\end{aligned}
\end{equation*}
as $L_h(z)=0$. Thus
\begin{equation*}
\begin{aligned}
\int_0^\infty \ln|S(i\omega)|w(z,\omega)d\omega =\pi\ln\left|\prod_k \frac{p_k+z}{p_k-z}\right|-\pi\ln\left|1+L_b(z)\right| \\
\end{aligned}
\end{equation*}
For the case of no unstable poles then the result follows similarly using $S_{ap}(s)=1$.
\end{proof}
Theorem \ref{th:zero_integral} shows that buffering is effective at improving fundamental limits for feedback with open-loop unstable zeros if $\left|1+L_b(z)\right|$ is sufficiently large. For $z\to0$ and lossless buffering ($a_{xx}=0$) then $L_b(z)\to 0$ as $C_b(0)= 0$. Small $z$ thus requires either very large $\sigma_y/\sigma_x$ or $a_{xx}>0$ (where $C_b(0)\ne 0$) for effective regulation.

Theorem \ref{th:zero_integral} uses the unstable open-loop poles of $L$. However, we can further assume that all unstable poles in $L$ are poles in $L_h$. In this case, the contributions from the open-loop feedback $L_h$ and buffering $L_b$ separate on the RHS of \eqref{eq:zero_integral}. However, the effectiveness of the limit reduction from buffering depends upon the location of the feedback zero $z$.

Although Theorem \ref{th:zero_integral} is specifically for feedback and buffering, the approach can also be used for other cases with multiple loops where only one loop has an unstable zero.

\subsection{Constraints on the Peak of the Sensitivity Function}

We next study the constraints on the peak of the sensitivity function for the case where the feedback loop has a unstable zero, but the buffering loop has only stable zeros. 

\begin{theorem}\label{th:zero_peak}
Assuming proper, rational $C_h(s)$ and $w_p(s)$, stable $w_pS$ and closed-loop stability, then
\begin{equation}
\|w_p(s)S(s)\|_\infty\ge \left|\frac{w_p(z)}{1+L_b(z)}\right|\prod_k \left|\frac{p_k+z}{p_k-z}\right|
\end{equation}
where $L(s)= L_b(s)+L_h(s)$, $L_b(s)=G_yC_b \sigma_y\bar{y}$, $L_h(s)=GC_h$, $z$ represents the unstable zero of $L_h$, $p_k$ represents the unstable poles of $L$ and $z$ is distinct from all $p_k$. If there are no RHP poles then 
\begin{equation*}
\|w_p(s)S(s)\|_\infty\ge \left|\frac{w_p(z)}{1+L_b(z)}\right|
\end{equation*}
\end{theorem}
\begin{proof}
We follow a modification of the proof in \cite{DOYFT90}. If there are no RHP poles then from the maximum modulus principle \cite{DOYFT90} and $L_h(z)=0$, we have
\begin{equation*}
\begin{aligned}
\left|\frac{w_p(z)}{1+L_b(z)}\right|&=\left|w_p(z)S(z)\right|\le \sup_{Re s\ge 0} w_p(s)S(s)=\|w_p(s)S(s)\|_\infty\\
\end{aligned}
\end{equation*}
If there are RHP poles then using factorization $S=S_{ap}S_{mp}$ (see Proof of Theorem \ref{th:zero_integral}), we have
\begin{equation*}
\begin{aligned}
\|w_p(s)S(s)\|_\infty&=\|w_p(s)S_{mp}(s)\|_\infty\\
&\ge\left| w_p(z)S_{mp}(z)\right|\\
&=\left|\frac{w_p(z)}{1+L_b(z)}\right|\prod_k \left|\frac{p_k+z}{p_k-z}\right|
\end{aligned}
\end{equation*}
\end{proof}
As above, Theorem \ref{th:zero_integral} shows that buffering is effective at improving fundamental limits of feedback with open-loop unstable zeros if $\left|1+L_b(z)\right|$ is sufficiently large.

\section{Example: Glycolysis}\label{sect:exam}

In this section, we apply our framework to the analysis of glycolytic regulation with creatine phosphate buffering. Glycolysis is an interesting example as it is both core to energy metabolism \cite{BERTS07} and the autocatalytic nature of its biochemical reactions results in the primary feedback mechanism having unstable open-loop poles and zeros \cite{CHABD11}. Here, we show the effect of creatine phosphate on reducing the fundamental limits for glycolytic regulation.

\subsection{Model of Glycolysis with ATP Buffering}

We use a model of glycolysis based on that proposed in \cite{CHABD11}, with the addition of buffering \cite{HANAPS17}. Consider 
\begin{equation*}
\begin{aligned}
\dot{y}=&-q\underbrace{f(y,u_h)}_{\text{PFK}}+(q+1)\underbrace{w(y)z}_{\text{PK}}
-\underbrace{g_y(y,x)+g_x(x,y)}_{\text{Buffering}}-\underbrace{V_y (1+d_y)}_{\text{ATP Demand}}\\
\dot{z}=&\underbrace{f(y,u_h)}_{\text{PFK}}-\underbrace{w(y)z}_{\text{PK}}\\
\dot{x}=&\underbrace{g_y(y,x)-g_x(x,y)}_{\text{Buffering}}
\end{aligned}
\end{equation*}
where $y$ represents ATP concentration, $z$ is a lumped variable representing intermediate metabolites, and $x$ represents the pCr (creatine phosphate) concentration. $g_y(y,x)$ and $g_x(x,y)$ are the forward and reverse ATP$~\leftrightarrow$~pCr reactions respectively, $f$ represents the flux through phosphofructokinase (PFK), $w(y)z$ represents the flux through pyruvate kinase (PK), $q$ represents the stoichiometric "investment-return" ratio associated with autocatalysis and $V_y$ represents the nominal ATP demand.

The nominal steady state occurs when
\begin{equation*}
\begin{aligned}
&f(\bar y,\bar{u}_h)=\bar{w}\bar z=V_y\\
&g_y(\bar y,\bar x)=g_x(\bar x,\bar y)+r_x(\bar x)
\end{aligned}
\end{equation*}
where $\bar{w}=w(\bar{y})$.

 The linearization of the three-state model is
\begin{equation*}
\begin{aligned}
\Delta\dot{y}=&((q+1)\alpha_w\bar z-q\alpha_f)\Delta y+(q+1)\bar{w}\Delta z\\
&-q\beta_f \Delta u_h-\sigma_y\Delta y+\sigma_x\Delta x-V_y d_y\\
\Delta\dot{z}=&(\alpha_f-\alpha_w\bar z)\Delta y-\bar{w}\Delta z+\beta_f \Delta u_h\\
\Delta\dot{x}=&\sigma_y\Delta y-\sigma_x\Delta x
\end{aligned}
\end{equation*}
where
\begin{equation*}
\begin{aligned}
\sigma_y&=\frac{\partial}{\partial y}[g_y-g_x], &\sigma_x&=\frac{\partial}{\partial x}[g_x-g_y], &\beta_f=\frac{\partial f}{\partial u_h}\\
\alpha_f&=\frac{\partial f}{\partial y},&\alpha_w&=\frac{\partial w}{\partial y}\\
\end{aligned}
\end{equation*}
Matching with \eqref{eq:min_mod_lin}, we have
\begin{equation*}
\begin{aligned}
A_{yy}&=(q+1)\alpha_w\bar z-q\alpha_f,& A_{yz}&=(q+1)\bar{w},&B_{yh}&=-q\beta_f,\\
A_{zy}&=\alpha_f-\alpha_w\bar z,& A_{zz}&=-\bar{w},&B_{zh}&=\beta_f,\\
a_{xx}&=0
\end{aligned}
\end{equation*}

\subsection{Transfer Functions}
Scaling the feedback input by setting $\hat p=f(\bar y,\bar{u}_h)=V_y$, we have the transfer functions
\begin{equation*}
\begin{aligned}
G_y&=\frac{1}{\bar{y}}\frac{s+\bar{w}}{s^2+(q\alpha_f-(q+1)\alpha_w\bar{z}+\bar{w})s-\bar{w}\alpha_f}\\
G_h&=-V_y\beta_f\frac{qs-\bar{w}}{s+\bar{w}},\quad \hat G_z=\frac{(q+1)\bar{w}}{s+\bar{w}},\quad \hat d_y=-V_y\\
\end{aligned}
\end{equation*}
Thus we have
\begin{equation*}
\begin{aligned}
L_b&=\frac{s+\bar{w}}{s^2+(q\alpha_f-(q+1)\alpha_w\bar{z}+\bar{w})s-\bar{w}\alpha_f}\frac{\sigma_ys}{s+\sigma_x}\\
L_h&=-\frac{V_y}{\bar{y}}\beta_fq\frac{s-\frac{\bar{w}}{q}}{s^2+(q\alpha_f-(q+1)\alpha_w\bar{z}+\bar{w})s-\bar{w}\alpha_f}C_h\\
\end{aligned}
\end{equation*}
We can observe that $G_y$ and $G_yG_h$ have the same poles, while the zero of $G_y$ is in the LHP and the zero of $G_yG_h$ is in the RHP. Depending upon the parameters, the open-loop feedback poles can be in the RHP or LHP \cite{CHABD11}.

\subsection{Fundamental Constraints}

The unstable zero in $L_h$ introduces a constraint on feedback regulation. However, we can use Theorem \ref{th:zero_integral} and \ref{th:zero_peak} to show that the buffering can reduce this limit if $|1+L_b(z)|$ is sufficiently large. The unstable zero of $L_h$ is at $z=\bar{w}/q$, and so we have
\begin{equation*}
L_b(z)=\frac{\sigma_y\bar{w}}{\bar{w}+\sigma_x q} \frac{q(q+1)}{\bar{w}+(q\alpha_f-(q+1)\alpha_w\bar{z}+\bar{w})q-\alpha_fq^2}
\end{equation*}
We can observe that for $\sigma_x\gg\bar{w}/q$, where the buffer is faster than the dynamics of the unstable zero then
\begin{equation*}
C_b(z)=\frac{\sigma_y\bar{w}}{\bar{w}+\sigma_x q}\approx\frac{\sigma_y}{\sigma_x}\frac{\bar{w}}{q}
\end{equation*}
and so the regulatory limits in Theorem \ref{th:zero_integral} and \ref{th:zero_peak} can be reduced by increasing the buffer equilibrium ratio $\sigma_y/\sigma_x$ \cite{HANAPS17}, where
\begin{equation*}
\left|\frac{1}{1+L_b}\right|\approx\left|\frac{1}{1+\frac{\sigma_y}{\sigma_x }\frac{\bar{w}}{q}G_y(z)}\right|
\end{equation*}
If the $z=\bar{w}/q$ is small then $\sigma_y/\sigma_x$ is required to be very large for the buffer to be effective.

For $\sigma_x\ll\bar{w}/q$, where the time to buffering equilibrium is slower than the dynamics of the unstable zero, then
\begin{equation*}
C_b(z)=\frac{\sigma_y\bar{w}}{\bar{w}+\sigma_xq}\approx\sigma_y
\end{equation*}
and so the regulatory limits in Theorem \ref{th:zero_integral} and \ref{th:zero_peak} can be reduced by increasing $\sigma_y$, where
\begin{equation*}
\left|\frac{1}{1+L_b}\right|\approx\left|\frac{1}{1+\sigma_yG_y(z)}\right|
\end{equation*}

\begin{figure}
\centering
    \includegraphics[width=\columnwidth]{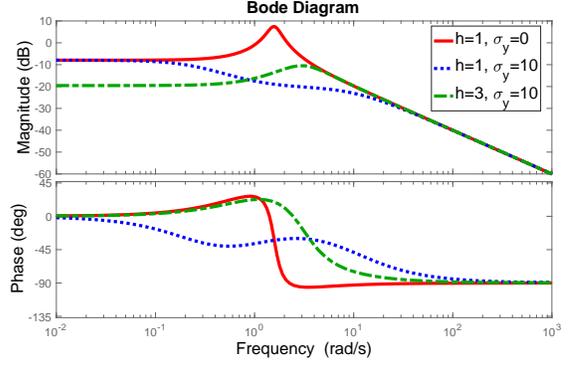}
\caption{Bode Plot of Glycolysis for $D_y$ to $Y$. The parameters are $w=1$, $q=1$, $ \alpha_w=-1$, $z=1$, $\alpha_f=1$, $\sigma_x=1$, $\beta_f=3.5$ and proportional feedback $C_h=h$.}
\label{fig:Bode}
\end{figure}

\begin{figure}
\centering
    \includegraphics[width=\columnwidth]{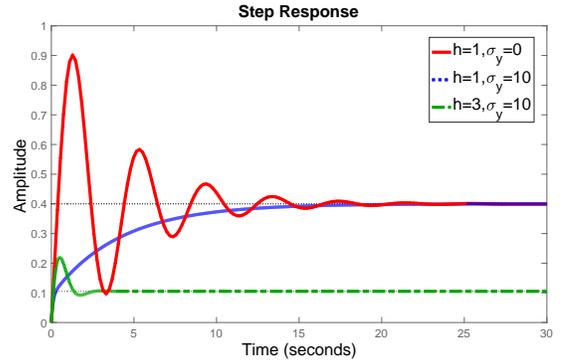}
\caption{Simulation of Glycolysis with  $w=1, q=1,\alpha_w=-1,z=1,\alpha_f=1,\sigma_x=1, \beta_f=3.5, \bar{y}=1$, proportional feedback $C_h=h$, and a unit step disturbance at $d_y$.}
\label{fig:Sim}
\end{figure}

Figures \ref{fig:Bode} and \ref{fig:Sim} shows the effect of increasing the buffering via an increase in the forward buffering kinetic constant $\sigma_y$ and proportional feedback $C_h=h$. Buffering can be seen to attenuate the oscillations from feedback, and also allow a stronger feedback gain without instability.

\section{Conclusion}

In this paper, we introduced a frequency-domain framework for analyzing buffer-feedback systems in cellular regulation. We determined standard closed-loop transfer function, loop-shaping properties, and the effect of buffering on fundamental limits. We then applied the methodology to the analysis of glycolysis. The work is likely to have a wide range of applications for biological regulation as well as autocatalytic networks more generally.

\appendices

\section{Transfer Functions}

\subsection{Plant Transfer Function}

The frequency-domain representation for \eqref{eq:min_mod_ol} without buffering or feedback is
\begin{equation*}
\begin{aligned}
\bar{y}sY(s)&=\bar{y}A_{yy}Y(s)+A_{yz}Z(s)+B_{yh}\hat{p}U_h(s)+\hat{g}U_b(s)+\hat d_yD_y(s)\\
sZ(s)&=\bar{y}A_{zy}Y(s)+A_{zz}Z(s)+B_{zh}\hat{p}U_h(s)+\hat d_zD_z(s)\\
\end{aligned}
\end{equation*}
which can be rearranged as
\begin{equation*}
\begin{aligned}
(s-A_{yy})\bar{y}Y&=A_{yz}Z+B_{yh}\hat{p}U_h+\hat{g}U_b+\hat d_yD_y\\
(sI_{n-1}-A_{zz})Z&=A_{zy}\bar{y}Y+B_{zh}\hat{p}U_h+\hat d_zD_z\\
\end{aligned}
\end{equation*}
giving
\begin{equation*}
\begin{aligned}
A_{yz}Z&=\hat G_zA_{zy}\bar{y}Y+\hat G_zB_{zh}\hat{p}U_h+\hat G_z\hat d_zD_z\\
\hat G_z &= A_{yz}(sI_{n-1}-A_{zz})^{-1}\\
\end{aligned}
\end{equation*}
Substituting, we obtain
\begin{equation*}
\begin{aligned}
(s-A_{yy}-\hat G_zA_{zy})\bar{y}Y=&(B_{yh}+\hat G_zB_{zh})\hat{p}U_h\\
&+\hat{g}U_b+\hat d_yD_y+\hat G_z\hat d_zD_z\\
\end{aligned}
\end{equation*}
Thus, we can write 
\begin{equation*}
\begin{aligned}
\bar{y}Y&=GU_h+\hat G_y\hat{g}U_b+\hat G_y\hat d_yD_y+\hat G_y\hat G_z\hat d_zD_z\\
\hat G_h&=B_{yh}+\hat G_zB_{zh},\quad \hat G_y =(s-A_{yy}-\hat G_zA_{zy})^{-1}\\
\end{aligned}
\end{equation*}
Setting
\begin{equation*}
\begin{aligned}
&\hat d_{yz}=\text{Diag}({\hat{d}_y,\hat{d}_z}),\quad D_{yz}(s)=\left(D_y,\, D_z\right)^T\\
&\hat G_d(s)=[1,\, \hat G_z],\quad \hat G_h=(B_{yh}+\hat G_zB_{zh})\\
\end{aligned}
\end{equation*}
then
\begin{equation}\label{eq:ol_app}
\begin{aligned}
Y(s)=&\frac{\hat{p}}{\bar{y}}\hat G_y\hat G_hU_h+\frac{\hat{g}}{\bar{y}}\hat G_yU_b+\frac{\hat d_{yz}}{\bar{y}}\hat G_y\hat G_dD_{yz}\\
\end{aligned}
\end{equation}

\subsection{Closed Loop}

We use the open-loop transfer function to determine the buffered, closed-loop transfer functions, i.e., feedback and buffering in \eqref{eq:ol_tf}. Comparing 
\eqref{eq:min_mod_ol} and \eqref{eq:min_mod_lin}, we have
\begin{equation}\label{eq:u_b}
\begin{aligned}
u_b=-\sigma_yy+\sigma_xx-d_b
\end{aligned}
\end{equation}
Transforming \eqref{eq:min_mod_lin} and \eqref{eq:u_b} results in
\begin{equation}\label{eq:tf_work}
\begin{aligned}
\hat{g}U_b=&-\sigma_y\bar{y}Y+\sigma_x\bar{y}X-\hat d_bD_b\\
s\bar{y}X=&\sigma_y\bar{y}Y-\sigma_x\bar{y}X-a_{xx}\bar{y}X+\hat d_x D_x+\hat d_bD_b\\
\end{aligned}
\end{equation}
where $X(s)=\mathcal{L}\{\Delta x/\bar{y}\}$.
Rearranging the second equation in \eqref{eq:tf_work}, and multiplying by $\sigma_x$, we have
\begin{equation*}
\begin{aligned}
\sigma_x \bar{y}X&=\frac{\sigma_x}{s+\sigma_x+a_{xx}}\sigma_y\bar{y}Y+\frac{\sigma_x}{s+\sigma_x+a_{xx}}(\hat d_xD_x+\hat d_bD_b)\\
&=(1-\hat{C}_b)\sigma_y\bar{y}Y+(1-\hat{C}_b)(\hat d_xD_x+\hat d_bD_b)\\
\end{aligned}
\end{equation*}
where
\begin{equation*}
\begin{aligned}
\hat{C}_b&=\frac{s+a_{xx}}{s+\sigma_x+a_{xx}},\quad 1-\hat{C}_b=\frac{\sigma_x}{s+\sigma_x+a_{xx}}
\end{aligned}
\end{equation*}
Thus 
\begin{equation*}
\begin{aligned}
\hat{g}U_b=&-\hat{C}_b\sigma_y\bar{y}Y+(1-\hat{C}_b)\hat d_xD_x -\hat{C}_b\hat d_bD_b\\
\end{aligned}
\end{equation*}
Substituting into \eqref{eq:ol_app}, we have
\begin{equation*}
\begin{aligned}
\bar{y}Y=&\hat G_y\hat G_h\hat{p}U_h-\hat G_y\hat{C}_b\sigma_y \bar{y}Y+\hat G_y(1-\hat{C}_b)\hat d_xD_x\\
&+\hat G_y\hat G_d\hat d_{yz}D_{yz}-\hat G_y\hat{C}_b\hat d_bD_b\\
\end{aligned}
\end{equation*}
Setting $U_h=-C_hY$ and $\hat{C}_h=C_h\hat{p}/\bar{y}$, we have the closed loop
\begin{equation*}
\begin{aligned}
&(1+\hat G_y\hat C_b\sigma_y+\hat G_y\hat G_h\hat{C}_h)\bar{y}Y\\
&=\hat G_y\hat G_d\hat d_{yz}D_{yz}+\hat G_y(1-\hat{C}_b)\hat d_xD_x-\hat G_y\hat C_b\hat d_bD_b\\
\end{aligned}
\end{equation*}
Thus
\begin{equation*}
\begin{aligned}
Y=&\frac{\hat G_y\hat G_d}{1+\hat G_y\hat C_b\sigma_y+\hat G_y\hat G_h\hat C_h}\frac{\hat d_{yz}}{\bar{y}}D_{yz}\\
&+\frac{\hat G_y(1-\hat C_b)}{1+\hat G_y\hat C_b\sigma_y+\hat G_y\hat G_h\hat C_h}\frac{\hat d_x}{\bar{y}}D_x\\
&-\frac{\hat G_y\hat C_b}{1+\hat G_y\hat C_b\sigma_y+\hat G_y\hat G_h\hat C_h}\frac{\hat d_b}{\bar{y}}D_b
\end{aligned}
\end{equation*}

%%%%%%%%%%%%%%%%%%%%%%%%%%%%%%%%%%%%%%%%%%%%%%%%%%%%%%%%%%%%%%%%%%%%%%%%%%%%%%%%
\section*{ACKNOWLEDGMENTS}

Edward Hancock gratefully acknowledges the donation from Judith and David Coffey. 

\bibliographystyle{plain}
\bibliography{Bio_Bibliography}

\end{document}